\documentclass[12pt,a4paper,reqno,oneside]{amsart}
\usepackage{t1enc}
\usepackage{times}
\usepackage{amssymb}
\usepackage[mathscr]{euscript}
\usepackage{graphicx}
\usepackage{hyperref}
\usepackage{color}
\usepackage{float}
\usepackage{epstopdf}

\newtheorem{thm}{Theorem}[section]

\newtheorem{prop}[thm]{Proposition}
\newtheorem{lem}[thm]{Lemma}
\newtheorem{cor}[thm]{Corollary}

\theoremstyle{remark}

\theoremstyle{definition}
\newtheorem{defn}[thm]{Definition}

\newcommand*{\rom}[1]{\expandafter\@slowromancap\romannumeral #1@} 
\renewcommand{\phi}{\varphi} 
\newcommand{\E}{\mathrm{E}} 

\newcommand{\nooutput}[1]{}
\newcommand{\sign}{\mathrm{sign}}
\DeclareMathOperator\vol{vol} 

\begin{document}
	
	\title{Explicit Local density bounds for It\^o-processes with irregular drift.}
	\date{\today}
	
	
	\author[Kr\"uhner]{Paul Kr\"uhner}
	\address[Paul Kr\"uhner]{\\
		FAM - financial and actuarial mathematics \\
		Technical University Vienna\\
		Wiedner Hauptstrasse 8-10\\
		1040 Vienna, Austria}
	\email[]{paul.kruehner@fam.tuwien.ac.at}
	
	\author[Xu]{Shijie Xu}
	\address[Shijie Xu]{Institute for Financial and Actuarial Mathematics, University of Liverpool}
	\email[]{ShijieXu@liverpool.ac.uk}

	
	\thanks{Financial support by the Austrian Science Fund (FWF) under grant P28661 is 
		gratefully acknowledged.} 
	
	\begin{abstract}
		We find explicit upper bounds for the density of marginals of continuous
		diffusions where we assume that the diffusion coefficient is constant
		and the drift is solely assumed to be progressively measurable and
		locally bounded. In one dimension we extend our result to the case that
		the diffusion coefficient is a locally Lipschitz-continuous function of
		the state. Our approach is based on a comparison to a suitable doubly
		reflected Brownian motion whose density is known in a series representation.
		

	\end{abstract}
	
	\maketitle
	
	\section{Introduction}
	The analysis of stochastic differential equations (SDEs) and in particular the analysis 
	of the law of their solutions has been a research topic of great interest. For SDEs 
	driven by a $d$-dimensional Brownian motion, one is interested in conditions imposed on 
	the coefficients of the SDE ensuring that the density of its solution is absolutely 
	continuous with respect to the Lebesgue measure at any time. In addition, one strives for 
	explicit bounds for Lebesgue density.\\
	\indent In order to derive these kinds of results, one usually imposes regularity 
	assumptions on the coefficients of the SDE. It appears that there are three different 
	approaches
	that are used to derive 
	properties of the density of the SDE's solution from regularity 
	assumptions imposed on the SDE's coefficients. The probably most noted approach is 
	Malliavin calculus. Malliavin \cite{malliavin.78} himself has provided smoothness and 
	non-degeneracy conditions on the coefficients of the SDE implying the existence as well 
	as smoothness and boundedness properties of the solution's density. Another approach that 
	is based on a stochastic calculus of variations has been proposed by Bouleau \& Hirsch 
	\cite{bouleau.hirsch.86}. There, the authors have used Dirichlet forms and a 
	limit procedure to derive absolute continuity of the finite-dimensional laws of solutions 
	to SDEs. Instead of using variational calculus, e.g.\ Malliavin calculus, some authors exploit techniques from control theory. Ba\~nos \& Kr\"uhner \cite{banos2016optimal} identify a worst-case scenario SDE among a family of SDEs. More precisely, they have proven that the density of this worst-case SDE, which is well known in the literature, dominates the densities of the other SDE-solutions. As a result, Ba\~nos \& Kr\"uhner derive optimal density bounds for the densities of solutions to SDEs with \emph{general progressively measurable, bounded} drift coefficient. Besides density estimation, some author also investigate the regularity of the density.  Hayashi, Kohatsu-Higa and Y{\^u}ki \cite{hayashi2013local} prove that SDEs with bounded H\"older-continuous drift and smooth elliptic diffusion coefficients admit H\"older-continuous densities at any time. 
	
    All the results mentioned above provide \emph{global bounds} to the density of the SDE's solution. Qian \& Xu \cite{qian2018optimal} has a description of the optimal local bound for the density under some local boundedness condition on the drift for Markovian processes. To be exact, the authors investigate Markovian diffusion processes on the domain with normal inward reflection, bounded drift and constant diffusion coefficient. They find the description of the optimal density bound which in the end for any point they identify the corresponding drift coefficient which produces a process that attains the bound in this point. 
     De Marco \cite{de2011smoothness} addresses more general SDE's and provides \emph{local existence results}. The author shows local smoothness of densities on an open domain under the usual condition of ellipticity and that the coefficients are smooth on such domain. 

	\indent We find a local upper bound which is not the sharpest as in \cite[Theorem 1]{qian2018optimal} but it is an explicit upper bound, cf.\ Proposition \ref{p:main inequality} below. Also, compared to \cite{qian2018optimal}, our process is non-Markovian and has no restriction on its domain. We only require local-Lipschitz for the diffusion coefficient which is weaker than the regularity condition used in \cite{hayashi2013local}. The authors also provide the smoothness of the density. Compared to De Marco's result \cite{de2011smoothness}, our estimated upper bound is explicit and we have no regularity requirement on the drift coefficient, while in multi-dimension cases our diffusion coefficient should be constant and we obtain no regularity of the density.   Our goal is to provide upper bounds for the density of the solution to an SDE on a local boundedness assumption of the drift coefficient. In order to establish our main result, we provide a link between the density of an It\^o-process and the transition density of a doubly reflected Brownian motion. Since the latter is known in closed form, we can exploit this representation and derive various local bounds for the density of an It\^o-process with a constant diffusion coefficient. These bounds only depend on the local behaviour of the drift coefficient. This approach allows us to reproduce the main result of Ba\~nos \& Kr\"uhner \cite{banos2016optimal}. In particular, an application of the It\^o-Tanaka formula allows us to extend the main result to It\^o-processes with a Lipschitz-continuous diffusion coefficient.

	\indent The paper is structured as follows. In Section \ref{section__main_results}, we present our 
	main result, Theorem \ref{t:main result}, as well as several corollaries. Section \ref{sec__proofs} provides proof of the main 
	theorem as well as an explicit representation of the transition density for a doubly reflected Brownian 
	motion. The latter enables us to derive the density bounds. Finally, we gather all the auxiliary and technical results in Appendix \ref{sec__Normal_expectations_and_estimates}.

	\subsection{Notations}
	Throughout the paper $W$ 
	is a standard $(\mathcal F_t)_{t\geq 0}$-Brownian motion on the 
	stochastic basis $(\Omega,\mathcal A,(\mathcal F_t)_{t\geq 0},P)$ where we assume that $(\mathcal F_t)_{t\geq 0}$ is right-continuous and complete. Besides, $\mathbb R_{+}:=[0,\infty)$ and for all $x\in\mathbb R^d$ we 
	denote the uniform norm by $\|x\|:=\max\{|x_j|:j=1,\dots,d\}$, the Euclidean norm by $|x|:=({\sum_{j=1}^d x_j^2})^{1/2}$, and the generalised signum function by $\sign(x):=1_{\{x\neq 0\}}x/\vert x\vert$. We denote by $e_j$ the $j$-th standard basis vector in $\mathbb R^d$, i.e.\ $j=1,\dots, d$ and $e_j(k)=1_{\{j=k\}}$. Furthermore, $\Phi$ (resp.\ $\phi$) denotes the distribution (resp.\ density) function of the standard normal law. For a Borel set $B\subseteq\mathbb R^d$ we denote by $\vol(B)$ the Lebesgue measure of $B$. We denote the open Euclidean ball with radius $\epsilon>0$ around some point $y\in\mathbb R^d$ by $B_{\epsilon}(y)$ and the open $\|\cdot\|$-ball by $B_{\epsilon,\infty}(y):=\times_{j=1}^d(y_j-\epsilon,y_j+\epsilon)$. Further notations are used as in \cite{jacod2013limit}.
	
	\section{Main statement and consequences}\label{section__main_results}
	The goal of this paper is to show that It\^o-processes with locally bounded drift and constant diffusion coefficient have locally bounded density at any time point $t>0$. In order to state this theorem we start to make the local boundedness precise.
	
	\begin{defn}
		Let $X$ be a $d$-dimensional It\^o-process, i.e.\ there is an $\mathbb R^d$-valued, progressively measurable process $\beta$ with locally integrable paths and an $\mathbb R^{d\times n}$-valued progressively measurable process $\sigma$ with locally square integrable paths such that
		$$X(t) = x_0 + \int_0^t\beta(s) ds + \int_0^t \sigma(s) dW(s),\quad t\geq 0$$
		where $W$ is an $n$-dimensional standard Brownian motion and $x_0\in\mathbb R^d$. 
		
		We say that the \textit{$X$ has bounded drift while $X$ is in some open set $U\subseteq\mathbb R^d$} if there is a constant $C>0$ such that
		$$ ||\beta_t|| 1_{\{X(t)\in U\}} \leq C $$
		for any $t\geq 0$, $P$-a.s.
		
		We say that $X$ has \textit{locally bounded drift} if $X$ has bounded drift on any bounded open set.
		
		We say that $X$ has \textit{non-degenerate diffusion coefficient} if $\sigma_t\sigma^{\top}_t$ is positive definite for any $t\geq 0$.
	\end{defn}
	
	As can be seen immediately from the definition, the bound $C$ is only effective as long as the process stays inside the set $U$. In the special case that $\beta(t) = b(X(t))$ for some locally bounded measurable function $b:\mathbb R^d\rightarrow \mathbb R^d$ we see that $X$ has locally bounded drift.
	
	\begin{thm}\label{t:main result}
		Let $X$ be a $d$-dimensional It\^o-process with constant, deterministic and non-degenerate diffusion coefficient. Assume that $X$ has bounded drift while $X$ is in some open set $U\subseteq\mathbb R^d$. Let $t>0$.
		
		Then 
		$$\rho_t(x) := \limsup_{\epsilon\rightarrow 0}\frac{P(||X(t)-x||<\epsilon) }{\vol(B_{\epsilon,\infty}(0))}, \quad x\in U$$ 
		is locally bounded.
		
		Moreover, $\rho_t$ is a version of the density of $X(t)$ on $U$, i.e.\ $P(X(t) \in A) = \int_A \rho_t(x)dx$ for any Borel-set $A\subseteq U$.
		
		In particular, if $X$ has locally bounded drift, then $X(t)$ has a locally bounded version of its density.
	\end{thm}

	The above theorem does in fact follow easily from the following more technical statement which contains an explicit upper bound for the density.
	\begin{prop}\label{p:main inequality}
		Let $X$ be an It\^o-process where the diffusion coefficient is constant equal to the identity matrix on $\mathbb R^d$ and assume that the drift of $X$ is bounded by $C$ while $X$ is in some open set $U\subseteq \mathbb R^d$. We define 
		$\rho_t(x) := \limsup_{\epsilon\rightarrow 0}\frac{P(||X(t)-x||<\epsilon )}{\vol(B_{\epsilon,\infty}(0)}\in[0,\infty]$ for $t>0$, $x\in\mathbb R^d$. Let $x\in U$ and $l>0$ such that $B_{l,\infty} (x)\subseteq U$.
		
		Then we have
		\begin{align*} 
			\rho_t(x) &\leq  \prod_{j=1}^d\left(\frac{C\exp(-2Cl)}{1-\exp(-2Cl)} + \frac{\phi(z_j)}{\sqrt{t}} + C \Phi(z_j) + e^{Ca_j-C^2t/2}\frac{(3+a_jC)^2}{ltC^2}\right)
		\end{align*}
		where $a_j:=\min\{l,|X_j(0)-x_j|\}$ and $z_j:=\sqrt{t}C-a_j/\sqrt{t}$ for $j=1,\dots,d$.
	\end{prop}
	\begin{proof}
		This is the joint conclusion of Propositions \ref{p:maximality} and \ref{p:Improved}.
	\end{proof}
	
	The result can be transferred to the situation where the drift coefficient of the SDE is globally bounded. The following Corollary reproduces Ba\~nos \& Kr\"uhner \cite{banos2016optimal}'s result.
	
	\begin{cor}
		Let $X$ be an It\^o-process where the diffusion coefficient is constant equal to the identity matrix on $\mathbb R^d$ and assume that the drift of $X$ is bounded by $C$. We define 
		$\rho_t(x) := \limsup_{\epsilon\rightarrow 0}\frac{P(||X(t)-x||<\epsilon )}{\vol(B_{\epsilon,\infty}(0))}\in[0,\infty]$ for $t>0$, $x\in\mathbb R^d$. 
		\begin{align*}
			\rho_t(x) &\leq  \prod_{j=1}^d 
			\left(\frac{\phi(z_j)}{\sqrt{t}} + C \Phi(z_j)\right) 
			\leq \left(\frac{1}{\sqrt{2\pi t}}+C\right)^d 
		\end{align*}
		where $(z_j)_{j\leq d}$ is given by $z_j := 
		\sqrt{t}C-|X_j(0)-x_j|/\sqrt{t}$. In particular, if $d=1$, then we have
		\begin{align*}
			\rho_t(x) &\leq \frac{\phi(z_1)}{\sqrt{t}} + C \Phi(z_1).
		\end{align*}
	\end{cor}
	\begin{proof}
		This is immediate from Proposition \ref{p:main inequality} by passing to the limit $l\rightarrow \infty$.
	\end{proof}
	
	Local boundedness also applies to finite dimensional marginals of the process $X$.
	\begin{cor}
		Let $X$ be an It\^o-process where the diffusion coefficient is constant and non-degenerate and assume that its drift is locally bounded. Let $0<t_1<\dots<t_N$ for some $N\in\mathbb N$. Then $ (X_{t_1},\dots,X_{t_N}) $ has a version of its density which is locally bounded.
	\end{cor}
	\begin{proof}
		Follows from Proposition \ref{p:main inequality}.
	\end{proof}

	The previous observation and It\^o-Tanaka's formula allow us to make a statement for solutions to $1$-dimensional SDEs where the drift coefficient is bounded and the diffusion coefficient is strictly positive and locally Lipschitz-continuous. While the existence of a solution to such SDEs is not guaranteed by these properties as solutions could e.g.\ explode, we simply assume that we are given a well-behaved solution.
	\begin{cor}
		Let $b,\sigma:\mathbb R\rightarrow\mathbb R$ such that $\sigma(y)>0$ for any $y\in\mathbb R$, $b$ is measurable and locally bounded and $\sigma$ is locally Lipschitz-continuous. Assume that there is an $\mathbb R$-valued It\^o-process $Y$ satisfying
		$$ dY_t = b(Y(t)) dt + \sigma(Y(t)) dW(t). $$
		Then $Y_t$ has a locally bounded version of its density for any $t>0$.
	\end{cor}
	\begin{proof}
		Define $F(y) := \int_0^y \frac{1}{\sigma(u)} du$ for any $y\in\mathbb R$. Note that $F$ is invertible and continuously differentiable with $F'(y) = \frac{1}{\sigma(y)}$. Since $\sigma$ is locally Lipschitz-continuous we find that $F'$ is absolutely continuous and a version of its absolutely continuous derivative satisfies $F''(y) = \frac{-\sigma'(y)}{\sigma(y)^2}$ where $\sigma'$ is a locally bounded version of the absolutely continuous derivative of $\sigma$. It\^o-Tanaka formula \cite[Chapter IV, Theorem 71]{protter.05} yields that $X_t := F(Y_t)$ satisfies
		$$ dX(t) = a(X(t)) dt + dW(t) $$
		where $a(x) := \frac{b(F^{-1}(x))}{\sigma(F^{-1}(x))} - \frac{\sigma'(F^{-1}(x))}{2}$ for $t\geq 0$. $a$ is a locally bounded function and, hence, $X_t$ has locally bounded density for any $t>0$. Consequently, Lemma \ref{l:Transfer of density} yields that $Y_t$ has locally bounded density for any $t>0$.
	\end{proof}

	\section{Proofs}\label{sec__proofs}
	
	Simply put, the proof of our main result relies on two steps. In the first step, we establish an upper bound for the density of  $X(t) = x + \int_0^t\beta(s) ds + W(t)$ in terms of the transition density of a $1$-dimensional, doubly reflected Brownian motion. As its transition density is known in terms of a series expansion, we can exploit this representation in order to derive a closed form upper bound for the density of $X$.
	
	First, we recall the definition of doubly reflected Brownian motion (DRBM) with drift.
	\begin{defn}
		A stochastic process $Z$ with continuous sample paths is a doubly reflected Brownian motion with drift $b\in\mathbb R$ on a compact interval $J$ if for any $f\in C^2(\mathbb R,\mathbb R)$ with $f'(x)=0$ for any boundary point $x\in J$ we have
		$$ M^f_t := f(Z_t) - \int_0^t \left(\frac{1}{2}f''(Z_s)+bf'(Z_s)\right)ds, \quad t\geq 0. $$
		is a martingale.
	\end{defn}
	
	We recall the existence and uniqueness statement for the DRBM.
	\begin{prop}\label{p:DRBM}
		Let $J$ be a compact interval with at least two points and $b\in\mathbb R$. Then (on some stochastic basis) there is a DRBM $Z$ with drift $b$ on $J$. If $Y$ is another DRBM with drift $b$ on $J$ (on a possibly different stochastic basis), then $Y$ has the same law as $Z$. 
		
		Moreover, one has (on a possibly enlarged probability space) that
		$$ dZ(t) = b dt + dW(t) + R(t) $$
		where $R$ has sample paths of finite variation and $dR$ is carried by the set $\{t\geq 0: Z(t)\text{ is at the boundary of }J\}$ and $W$ is some one-dimensional standard Brownian motion.
	\end{prop}
	\begin{proof}
		See \cite[Chapter 8, Theorem 1.1]{ethier.kurtz.86} for the uniqueness in law and \cite[Section 2.8.C, Exercise 8.9]{karatzas.shreve.98}.
	\end{proof}
	
	Before we start, let us fix some notation. For $l>0$ and $C\geq 0$, we denote by $Z^l$ a doubly reflected Brownian motion on $[0,l]$ with drift $-C$. In addition, we denote its transition density by $p_{l,t}(x,y)$, i.e.\ for any 
	Borel set $A\subseteq \mathbb R$ we have $P(Z^l(T) \in A| Z^l(t) = x) = \int_A p_{l,T-t}(x,y)dy$ (existence of transition density follows from \cite[p.\ 193, formula (13)]{veestraeten.04}. Let $(\Omega,\mathfrak A,(\mathcal F_t)_{t\geq 0},P)$ be a filtered probability space with $\mathcal F$ right-continuous and $\mathcal F_0$ containing all $P$-null sets. Further, let $W$ be an $\mathcal F$-Brownian motion.
	
	We need the following technical result which compares the probability that a $1$-dimensional It\^o-process with diffusion coefficient equal to $1$ is at a fixed time $t$ in an open interval to that of a doubly reflected Brownian motion with the same starting value but smaller drift.
	\begin{lem}\label{l:comparision1}
		Let $\beta$ be a progressively measurable $\mathbb R$-valued process with $|\beta(t)|\leq C$ for any $t\geq 0$. Let $0\leq z_0\leq x_0$ and
		\begin{align*}
			X(t) &:= x_0 + \int_0^t \beta(s) ds + W(t) + R_t(X), \\
			Z(t) &:= z_0 - Ct + W(t) + R_t(Z) - A(t)
		\end{align*}
		where $A$ is any continuous increasing progressively measurable process with $A(0)=0$ and $R(X)$ resp.\ $R(Z)$ are the respective upward reflection terms at zero for $X$ resp.\ $Z$, i.e.\ 
		\begin{align*}
			R_t(X) := \sup \left\{  \max\{0, -(x_0 + \int_0^u \beta(s) ds + W(u))\} : u\in [0,t] \right\}, \\
			R_t(Z) := \sup \left\{  \max\{0, -(z_0 - Cu + W(u) - A(u))\} : u\in [0,t] \right\}.
		\end{align*}
		
		Then $Z(t) \leq X(t)$ for any $t\geq 0$.
	\end{lem}
	\begin{proof}
		We will now inspect a single path and assume that $\omega$ is fixed for the remainder of the proof.
		Define $G:= \{ t\geq 0: X(t) < Z(t) \}$. By continuity $G$ is an open set. Assume for contradiction that $G$ is non-empty. Since $G$ is open it is the countable disjoint union of open intervals. Let $U$ be one of those intervals and define $t_0:=\inf(U)\geq 0$. Observe that $X(t_0)=Z(t_0)$ and that $X(t)<Z(t)$ for any $t\in U$. 
		
		Note that $dR_t(Z)$ is carried on the set $\{t\geq 0: Z(t)=0\}$ and $R_t {(\cdot)}$ is continuous in $t$. Since $0\leq X(t)<Z(t)$ for any $t\in U$ we find that $R_t(Z)-R_{t_0}(Z) = 0$ for any $t\in U$. Now, let $t\in U$. We have
		\begin{align*}
			0 &> X(t) - Z(t) \\
			&= (X(t) - X(t_0)) - (Z(t)-Z(t_0)) \\
			&= \int_{t_0}^t (\beta(s)+C) ds + (R_t(X)-R_{t_0}(X)) + (A(t)-A(t_0)) \\
			&\geq 0.
		\end{align*}
	\end{proof}    
	
	We extend Lemma \ref{l:comparision1} to $d$-dimension and show that the reflected processes are independent.
	\begin{lem}\label{l:comparisond}
		Let $X$ be a $d$-dimensional It\^o-process with diffusion coefficient constant equal to the identity matrix. Assume that the drift of $X$ is bounded while $X$ is in the set $B_{l,\infty}(x)$ where $l>0$, $x\in\mathbb R^d$ and we denote the corresponding constant by $C\geq 0$.
		Let $Y_1,\dots,Y_d$ be independent doubly reflected Brownian motions with drift $-C$ on $[0,l]$. Assume that 
		$ |Y_j(0)| \leq |X_j(0) - x_j|$ for any $j=1,\dots,d$.
		
		Then
		$$  P( \|X(t)-x\| \leq a) \leq P( \|Y(t)\| \leq a ) , \quad a\in (0,l].$$
	\end{lem}
	\begin{proof}
		For $j\in \{1,\dots,d\}$, by Tanaka's formula \cite[p.222, Theorem 1.2]{revuz.yor.99} we have
		\begin{align*}
			|X_j(t)-x_j|&= |X_j(0)-x_j| + \int_0^t\sign(X_j(s)-x_j)\beta_j(s) ds + B_j(t) + R_t(X_j)	
		\end{align*}
		for the standard Brownian motion $B_j(t)=\int_0^t \sign(X_j(s)-x_j)dW_j(s)$ and $R(X_j)$ is the respective upward reflection term which is carried on the set $\{t\geq 0 :X_j(t) = 0\}$.
		
		Let $Z^j$ be a doubly reflected Brownian motion on $[0,l]$ with starting point $|Y_j(0)|$, drift $-C$ and martingale part $B_j$. Proposition \ref{p:DRBM} yields that $P^{Z_j}=P^{Y_j}$. Lemma \ref{l:comparision1} states that
		$$ Z_j(t) \leq |X_j(t)-x_j|,\quad t\geq 0. $$
		Consequently, we find that
		\begin{align*}
			P( |X_j(t)-x_j| \leq a) & \leq P(Z_j(t) \leq a).
		\end{align*}
		Note that $B=(B_1,\dots,B_d)$ is a continuous martingale and we have $[B_j,B_k]_t = t 1_{\{j=k\}}$ for $j,k=1,\dots,d$. Hence, L\'evy's characterisation for Brownian motion \cite[Chapter I.4.54]{jacod2013limit} yields that $B$ is a $d$-dimensional standard Brownian motion. Consequently, $Z_1,\dots,Z_d$ are independent processes. We have
		\begin{align*}
			P( \|X(t) - x\| \leq a) & \leq P( \|Z(t)\| \leq a ) \\
			&= \prod_{j=1}^d P(Z_j(t) \leq a) \\ 
			&= \prod_{j=1}^d P(Y_j(t) \leq a) \\
			&= P( \| Y(t) \| \leq a ). 
		\end{align*}	   
	\end{proof}

	\begin{prop}\label{p:maximality}
		Let $X$ be a $d$-dimensional It\^o-process with diffusion coefficient constant equal to the identity matrix and $t>0$. Let $U\subseteq \mathbb R^d$ be open and assume that the drift of $X$ is bounded by $C\geq0$ while $X$ is in $U$. Define
		$$ \rho_t(x) := \limsup_{\epsilon\searrow0} \frac{P(\|X(t)-x\|\leq \epsilon)}{\vol(B_{\epsilon,\infty}(0))}\in[0,\infty],\quad x\in\mathbb R^d $$
		and we denote the transition density of a doubly reflected Brownian with drift $-C$ on $[0,l]$ over $t$ time units by $p_{l,t}$.
		
		Then we have
		$$ \rho_t(x) \leq \frac{1}{2^d}\prod_{j=1}^d p_{l,t}(a_j,0),\quad x\in U $$
		where $a_j = \min\{l, |X_j(0)-x_j|\}$ and $l>0$ such that $B_{l,\infty}(x)\subseteq U$.
	\end{prop}
	\begin{proof}
		Let $x\in \mathbb R^d$, $a_j = \min\{l, |X_j(0)-x_j|\}$ and $l>0$ such that $B_{l,\infty}(x)\subseteq U$.
		
		Let $Y^1,\dots,Y^d$ be independent doubly reflected Brownian motions with drift $-C$ and starting point $Y_j(0) = a_j$ for $j=1,\dots,d$. Then we have $|Y_j(0)| \leq |X_j(0)-x_j|$ for $j=1,\dots, d$ by construction. Lemma \ref{l:comparisond} yields that
		$$ P(\|X(t) - x\| \leq \epsilon) \leq P( \|Y(t)\| \leq \epsilon),\quad \epsilon\in (0,l]. $$
		
		Thus, we have
		$$ \rho_t(x) \leq \limsup_{\epsilon\searrow0} \frac{P(\|Y(t)\|\leq \epsilon)}{\vol(B_{\epsilon,\infty}(0))} = \limsup_{\epsilon\searrow0} \prod_{j=1}^d \frac{P(|Y_j(t)|\leq \epsilon)}{2\epsilon} = \frac{1}{2^d} \prod_{j=1}^d p_{l,t}(a_j,0) $$
		where the last limit superior is a limit because the transition density of the doubly reflected Brownian is continuous.
	\end{proof}

	Proposition \ref{p:maximality} links the transition density of the reflected Brownian motion to an upper bound for the density of an It\^o-process. Consequently, we focus our analysis on the transition density of a doubly reflected Brownian motion. The next result is a known expression for the density. It is adopted from \cite[p.\ 193, formula (13)]{veestraeten.04}.
	\begin{prop}\label{p:RBMwD2}
		Let $p$ be the transition density of a doubly reflected Brownian motion. Then $p$ is continuous in all its arguments and $p_{l,t}(x,0)\leq p_{l,t}(0,0)$. In particular, for all $x\in [0,l]$ and $t>0$ we have
		$$ p_{l,t}(x,0) = \frac{2C}{1-\exp(-2Cl)}+e^{Cx-C^2t/2}\frac{2}{l}\sum_{n=1}^{\infty} \left[f_{t,x}(n\pi/l)-g_{t,x}(n\pi/l)\right],$$
		where $f_{t,x}(z):=\frac{z^2\cos(zx)}{C^2+z^2}\exp(-tz^2/2)$ and $g_{t,x}(z):=\frac{Cz\sin(zx)}{C^2+z^2}\exp(-tz^2/2)$ for $z\in\mathbb R$.
	\end{prop}
	\begin{proof}
		The statement can be found in \cite[p.\ 193, formula (13)]{veestraeten.04}, where $d=l$, $c=0$, $\sigma=1$, $a=l/\pi$, $\mu=-C$.
	\end{proof} 
	
	This exact representation of the transition density $p$ allows us to derive upper bounds for it.

	\begin{prop}\label{p:Improved}
		For all $t,l>0$ and $x\in [0,l]$ we have
		\begin{align*}
			p_{l,t}(x,0) &\leq \frac{2C\exp(-2Cl)}{1-\exp(-2Cl)} + \frac{2}{\sqrt{t}}\phi(\sqrt{t}C-x/\sqrt{t}) + 2C \Phi(\sqrt{t}C-x/\sqrt{t}) 
			\\&\quad+ 2e^{Cx-C^2t/2}\frac{(3+xC)^2}{ltC^2}
			\\ &\leq \frac{1}{l} + \frac{2}{\sqrt{t}}\phi(\sqrt{t}C-x/\sqrt{t}) + 2C \Phi(\sqrt{t}C-x/\sqrt{t}) 
			\\&\quad+ 2e^{Cx-C^2t/2}\frac{(3+xC)^2}{ltC^2}.
		\end{align*}
	\end{prop}
	\begin{proof}
		This  is an immediate consequence of Proposition \ref{p:RBMwD2} and
		Corollary \ref{k:Final}.
	\end{proof}

	\appendix
	\section{Normal expectations and estimates}\label{sec__Normal_expectations_and_estimates}
	We start by calculating some normal expectations.
	\begin{lem}\label{l:normal expectation cos}
		Let $a>0$, $b \geq0$, and $Z$ be a standard normal random variable. Then 
		\begin{align*}
			\E\left[\frac{a^2\cos(bZ)}{a^2+Z^2}\right] &= a\sqrt{2\pi} e^{a^2/2}\left(\cosh(ab)-\frac12(e^{-ab}\Phi(a-b)+e^{ab}\Phi(a+b))\right) \\
			\E\left[\frac{Z\sin(bZ)}{a^2+Z^2}\right] &= -\sqrt{2\pi} e^{a^2/2}\left(\sinh(ab)+\frac1{2}(e^{-ab}\Phi(a-b)-e^{ab}\Phi(a+b))\right)
		\end{align*}
		hold, where $\Phi$ denotes the distribution function of the standard normal law.
	\end{lem}
	\begin{proof}
		We denote by $\rho(x) := \frac{a}{2}\exp(-a|x|)$, $x\in\mathbb R$, the density of a 
		$\mathbb{R}$-valued, Laplace distributed, random variable $Y$ with parameter $a$. 
		Besides, $X$ denotes a random variable with values $\{b,-b\}$, independent of $Y$ such 
		that $P(X=b)=\frac{1}{2}$ holds. Then we get $\E[e^{iuY}] = \frac{a^2}{a^2+u^2}$ as 
		well as $\E[e^{iuX}] = \cos(ub)$ for $u\in\mathbb R$. In particular, $W:=X+Y$ satisfies 
		$f(u) := \E[e^{iuW}] = \frac{a^2\cos(ub)}{a^2+u^2}$ for all $u\in\mathbb R$ and the 
		density of $W$ is given by $\rho_W(x) = \frac 
		a4\left(\exp(-a|x-b|)+\exp(-a|x+b|)\right)$, $x\in\mathbb R$.
		%
		
		Let $\phi$ be the density function of the standard normal random variable $Z$ and for $u\in\mathbb R$ we denote its characteristic function by $g(u) := \exp(-u^2/2) = \E[e^ {iu Z}]$. Then Plancherel's theorem \cite[Theorem 2.2.14]{grafakos.08} yields
		\begin{align*}
			\E\left[\frac{a^2\cos(bZ)}{a^2+Z^2}\right] &= \frac{1}{\sqrt{2\pi}} 
			\int_{-\infty}^{\infty} f(u)g(u)du 
			= \sqrt{2\pi} \int_{-\infty}^{\infty} \rho_Z(x)\phi(x) dx \\
			&= a\sqrt{2\pi} e^{a^2/2}\left(\cosh(ab)-\frac12(e^{-ab}\Phi(a-b)+e^{ab}\Phi(a+b))\right).
		\end{align*}
		Moreover, we have
		\begin{align*}
			\E\left[\frac{Z\sin(bZ)}{a^2+Z^2}\right] &= -\frac{1}{a^2}\partial_b \E\left[\frac{a^2\cos(bZ)}{a^2+Z^2}\right] \\
			&= -\sqrt{2\pi} e^{a^2/2}\left(\sinh(ab)+\frac1{2}\left(e^{-ab}\Phi(a-b)-e^{ab}\Phi(a+b)\right)\right).
		\end{align*}
	\end{proof}

	The next statement is a consequence of the previous result.
	\begin{cor}\label{k:Integral}
		Let $t,C> 0$ and $x\geq 0$. Then we have
		$$ 2\int_0^\infty \left[f_{t,x}(z)-g_{t,x}(z)\right] dz = 
		\sqrt{\frac{2\pi}{t}}e^{-\frac{x^2}{2t}} + 2\pi C 
		e^{-Cx+tC^2/2}\left(\Phi(\sqrt{t}C-x/\sqrt{t})-1\right), $$
		where $f_{t,x}(z) := \frac{z^2\cos(xz)}{C^2+z^2}e^{-tz^2/2}$ and $g_{t,x}(z) := \frac{Cz\sin(xz)}{C^2+z^2}e^{-tz^2/2}$ for $z\in\mathbb R$.
	\end{cor}

	Estimates for the integral are closely connected to estimates for the sum in which we are 
	actually interested. The error-term can be controlled by the Euler-Maclaurin formula. We 
	use the elementary estimate
	\begin{equation*}\label{equation__eq_1_elementary_estimate}
		F:=\left|\sum_{k=1}^Nf(k)-\int_0^Nf(y)dy\right| \leq \int_0^N |f'(y)|dy
	\end{equation*}
	which holds for any $C^1$-function. We first gather an inequality for the derivative of 
	the functions appearing in the preceding corollary.
	
	\begin{lem}\label{l:fprime inequality}
		Under the assumptions of Corollary \ref{k:Integral}, we have for all $z\geq0$
		$$ |f'_{t,x}(z)|+|g'_{t,x}(z)| \leq (4z+x^2zC^2+tz^3+3Cxz+tCxz^3) \frac{e^{-tz^2/2}}{C^2}.$$
	\end{lem}
	\begin{proof}
		For a function $h:\mathbb R\rightarrow\mathbb R$ and $u(z) := \frac{h(z)}{z^2+C^2}e^{-tz^2/2}$ we have
		$$ u'(z) = \frac{h'(z)-tzh(z)}{z^2+C^2}e^{-tz^2/2} - 
		\frac{2zh(z)}{(C^2+z^2)^2}e^{-tz^2/2}$$
		as well as
		$$ |u'(z)| \leq (|h'(z)|+tz|h(z)|+2|h(z)|/z)\frac{e^{-tz^2/2}}{z^2+C^2}. $$
		If we choose $h(z) = z^2\cos(xz)$, we get $f_{x,t}(z) = \frac{h(z)}{z^2+C^2}e^{-tz^2/2}$. 
		Moreover, $|\cos(xz)|\leq 1$ and $|\sin(xz)|\leq xz$ ensure that
		$$ |f_{x,t}'(z)| \leq (2z+x^2z^3+tz^3+2z)\frac{e^{-tz^2/2}}{z^2+C^2}\leq 
		\frac1{C^2}(2z+x^2zC^2+tz^3+2z)e^{-tz^2/2} $$
		holds. Similarly, for $h(z) = Cz\sin(xz)$, we get $g_{x,t}(z) = 
		\frac{h(z)}{z^2+C^2}e^{-tz^2/2}$ as well as
		$$ |g_{x,t}'(z)| \leq (3Cxz+tCxz^3)\frac{e^{-tz^2/2}}{C^2}.$$
		Finally, we can conclude
		$$ |f_{x,t}'(z)|+|g_{x,t}'(z)| \leq (4z+x^2zC^2+tz^3+3Cxz+tCxz^3)\frac{e^{-tz^2/2}}{C^2}. $$
	\end{proof}

	\begin{cor}\label{k:Final}
		Let $C> 0$, $x\geq 0$, $t>0$, and define $f_{t,x},g_{t,x}$ as in Corollary 
		\ref{k:Integral}. Then
		\begin{align*}
			I-F \leq \frac{2}{l}\sum_{n=1}^{\infty} \left[f_{t,x}(n\pi/l)-g_{t,x}(n\pi/l)\right] \leq I+F
		\end{align*}    
		holds for any $l>0$, where 
		\begin{align*}
			I &:= \frac{2}{\sqrt{2\pi t}}e^{-\frac{x^2}{2t}} + 2C 
			e^{-Cx+tC^2/2}\left(\Phi(\sqrt{t}C-x/\sqrt{t})-1\right), \\
			F &:= \frac{2(3+xC)^2}{ltC^2}.
		\end{align*}  
	\end{cor}
	\begin{proof}
		Due to Corollary \ref{k:Integral},
		$$ \frac{2}{l}\int_0^\infty \left[f_{t,x}(z\pi/l)-g_{t,x}(z\pi/l)\right] dz = I $$
		holds. Therefore, \eqref{equation__eq_1_elementary_estimate} and Lemma \ref{l:fprime 
			inequality} reveal 
		\begin{align*}
			&\left|\frac{2}{l}\sum_{n=1}^{\infty} \left[f_{t,x}(n\pi/l)-g_{t,x}(n\pi/l)\right] - 
			I\right| \leq \frac{2\pi}{l^2} 
			\int_0^\infty\left|f'_{t,x}(y\pi/l)-g'_{t,x}(y\pi/l)\right| dy \\
			&\qquad= \frac{2}{l} \int_0^\infty\left|f'_{t,x}(y)-g'_{t,x}(y)\right| dy \\
			&\qquad\leq\frac{2}{l} 
			\int_0^\infty(4z+x^2zC^2+tz^3+3Cxz+tCxz^3) \frac{e^{-tz^2/2}}{C^2} dz \\
			&\qquad= \frac{14+12Cx+2x^2C^2}{ltC^2} \\
			&\qquad\leq \frac{2(3+xC)^2}{ltC^2}.
		\end{align*} 
	\end{proof}

	Next, we revisit the Lebesgue differentiation theorem adapted to our situation. Basically, we start out with a random variable $X$, where existence of the density is unknown, but a suitable differential quotient is assumed to be locally bounded. Under this assumption, $X$ has Lebesgue density and a version of it is given by the aforementioned quotient.
	\begin{prop}\label{p:locmaxden}
		Let $X$ be an $\mathbb R^d$-valued random variable and define
		$$\rho(y):= \limsup_{\epsilon \searrow 0} \frac{P(X\in B_{\epsilon ,\infty}(y)) }{\vol(B_{\epsilon, \infty}(0))} \in [0,\infty].$$
		Assume that $\rho$ is locally bounded.
		
		Then $X$ has density and $\rho$ is a version of its density.
	\end{prop}
	\begin{proof}
		By localization, we may assume that $\rho$ is bounded by some $K\geq 0$. Define
		$$\mathcal M:=\{B\in\mathcal{B}:\nu(B)\leq 2K\vol(B)  \}$$
		where $\mathcal B$ is the Borel $\sigma$-algebra on $\mathbb R^d$.
		
		We show that $\mathcal M=\mathcal B$. Note that $\mathcal M$ is a monotone class in the sense of \cite[p.496, Theorem 4.1]{ethier.kurtz.86}. By assumption we find for any $y\in\mathbb R^d$, $\epsilon_y>0$ such that for any $\epsilon\in(0,\epsilon_y)$ we have $B_{\epsilon ,\infty}(y) \in \mathcal M$. If $A,B\in\mathcal M$ are disjoint, then $A\cup B \in \mathcal M$. The corresponding is true for countable families of disjoint elements in $\mathcal M$. Since $\mathcal M$ is a monotone class we also find $\overline{B_{\epsilon ,\infty}(y)} \in \mathcal M$. Any open set can be exhausted in measure (relative to $\nu$ and to $\vol$) by a countable disjoint union of such closed balls. Consequently, $\mathcal M$ contains all open sets. We find that $\mathcal M$ contains $\mathcal A_0 := \{ ([a_1,b_1)\times \dots \times [a_d,b_d)) \cap \mathbb R^d: a_1,\dots,a_d,b_1,\dots,b_d\in [-\infty,\infty]\}$. Since $\mathcal M$ is closed under disjoint union it also contains $\mathcal A:=\{ \bigcup_{j=1}^n C_j: n\in\mathbb N, C_1,\dots,C_n\in\mathcal A_0\text{ pairwise disjoint}\}$. Note that $\mathcal A$ is an algebra of sets. Observe that $\sigma(\mathcal{A})$ is the Borel $\sigma$-algebra $\mathcal B$. Consequently, the monotone class theorem \cite[p.496, Theorem 4.1]{ethier.kurtz.86} yields that 
		$$\mathcal{B}=\sigma(A)\subseteq M\subseteq \mathcal{B}.$$
		
		We find $\nu$ is absolutely continuous with respect to the Lebesgue measure. The Radon-Nikodym theorem \cite[p.422, Theorem 32.2]{billingsley2008probability} yields that $\nu$ has Lebesgue density $f$. The Lebesgue differentiation theorem \cite[p.87, Corollary 2.1.16]{grafakos.08} yields 
		$$f(y)=\lim_{\epsilon \searrow 0}\frac{\nu({B_{\epsilon, \infty}(y)})}{\vol(B_{\epsilon, \infty}(0))}=\rho(y)$$
		for Lebesgue almost any $y\in \mathbb R^d$. Hence, $\rho$ is a version of the density.
	\end{proof}

	Lipschitz-functions on $\mathbb R^d$ have some stability properties to carry over local bounded densities.
	\begin{lem}\label{l:Transfer of density}
		Let $Y$ be an $\mathbb R^d$ valued random variables, $F:\mathbb R^d\rightarrow\mathbb R^d$ be a locally Lipschitz-continuous function and assume that $F(Y)$ has locally bounded density.
		
		Then $Y$ has locally bounded density.
		
		Moreover, if $L$ is a global Lipschitz-constant for $F$ and $\rho$ a version of the density of $F(Y)$, then there is a version $\rho^Y$ of the density of $Y$ such that
		$$\rho^Y(y) \leq L^d \rho(F(y)), \quad y\in\mathbb R^d. $$
	\end{lem}
	\begin{proof}
		Let $\epsilon>0$ and $y\in\mathbb R^d$. Let $L$ be a Lipschitz-constant for $F$ on $B_{\epsilon}(y)$. We find that $F(B_{\epsilon,\infty}(y))\subseteq B_{L\epsilon,\infty}(F(y))$ and, hence, we have
		$$ \{Y\in B_{\epsilon,\infty}\} \subseteq \{ X(t) \in B_{L\epsilon,\infty}(F(y)) \}. $$
		Let $\rho$ be a locally bounded version of the density of $F(Y)$. Thus, we have. 
		\begin{align*}
			P(Y\in B_{\epsilon,\infty}(y)) &\leq P( F(Y) \in B_{L\epsilon,\infty}(F(y)) ) \\
			&= \int_{B_{L\epsilon,\infty}(F(y))} \rho(x) dx \\
			&= \int_{B_{\epsilon,\infty}(0)} \rho(Lx+F(y)) L^d dy
		\end{align*}     
		Consequently, we find from Lebesgue's differentiation theorem and the previous inequality that
		$$ \rho^Y(z) := \limsup_{\delta \searrow 0} \frac{P(Y\in B_{\delta,\infty}(z))}{\vol(B_{\delta,\infty}(z))} \leq L^d \rho(F(z)), \quad z\in B_{\epsilon,\infty}(y). $$
		Thus, $\rho^Y$ is locally bounded on $B_{\epsilon,\infty}(y)$. 
		
		Since we can make this construction for any $y\in \mathbb R^d$ we see that $\rho^Y$ is locally bounded.
		
		Proposition \ref{p:locmaxden} yields that $Y$ has density and $\rho^Y$ is a version of the density of $Y$. 
		
		Now assume that $L$ is a global Lipschitz-constant for $F$. Since $L$ can be chosen independently of the position we find that
		$$ \rho^Y(z) \leq L^d \rho(F(z)), \quad z\in\mathbb R^d. $$
	\end{proof}

	\bibliographystyle{alpha}
	
	\bibliography{main.bib}

\end{document}